\documentclass {amsart}
\usepackage{amssymb,amsthm,color, amscd, mathtools}
\usepackage{}

\title{\bf Torsion-free Abelian Groups revisited (2)}

\author {Phill Schultz}
\address[phill.schultz@uwa.edu.au]{The University of Western Australia}

\theoremstyle{plain}
\newtheorem{theorem}{Theorem}[section]
\newtheorem{corollary}[theorem]{Corollary}
\newtheorem{lemma}[theorem]{Lemma}
\newtheorem{proposition}[theorem]{Proposition}

\theoremstyle{definition}
\newtheorem{definition}[theorem]{Definition}
\newtheorem{notation}[theorem]{Notation}

\theoremstyle{remark}
\newtheorem{remark} [theorem]{Remark}

\begin{document}

\renewcommand{\leq}{\leqslant}
\renewcommand{\geq}{\geqslant}
\newcommand{\Aut}{\mathop{\mathrm{Aut}}\nolimits}
\newcommand{\End}{\mathop{\mathrm{End}}\nolimits}
\newcommand{\Ker}{\mathop{\mathrm{Ker}}\nolimits}
\newcommand{\Hom}{\mathop{\mathrm{Hom}}\nolimits}

\newcommand\bbQ{{\mathbb{Q}}}
\newcommand\bbZ{{\mathbb{Z}}}
\newcommand\bbN{{\mathbb{N}}}
\newcommand\bbT{{\mathbb{T}}}
\newcommand{\bbP}{\mathbb{P}}

\newcommand{\calQ}{{\mathcal Q}}
\newcommand{\calP}{\mathcal{P}}
\newcommand{\calC}{\mathcal{C}}
\newcommand{\calD}{\mathcal{D}}
\newcommand{\calA}{\mathcal{A}}
\newcommand{\calH}{\mathcal{H}}
\newcommand{\calS}{\mathcal{S}}
\newcommand{\calB}{\mathcal{B}}
\newcommand{\calX}{\mathcal{X}}
\newcommand{\calY}{\mathcal{Y}}
\newcommand{\calL}{\mathcal{L}}
\newcommand{\calM}{\mathcal{M}}
\newcommand{\calG}{\mathcal{G}}
\newcommand{\calT}{\mathcal{T}}
\newcommand{\calI}{\mathcal{I}}
\newcommand{\calJ}{\mathcal{J}}
\newcommand{\calE}{\mathcal{E}}
\newcommand{\calK}{\mathcal{K}}
\newcommand{\calF}{\mathcal{F}}
\newcommand{\calV}{\mathcal{V}}
\newcommand{\calU}{\mathcal{U}}

 \newcommand{\ov}{\overline}
\newcommand{\rank}{\operatorname{rank}}
\newcommand{\Coker}{\operatorname{Coker}}
\renewcommand{\Im}{\operatorname{Im}}
\newcommand{\type}{\operatorname{type}}
\newcommand{\lcm}{\operatorname{lcm}}
\newcommand{\Reg}{\operatorname{Reg}}
\newcommand{\Rep}{\operatorname{Rep}}
\newcommand{\Rind}{\operatorname{RegInd}}
\newcommand{\Jon}{\operatorname{Jon}}
\newcommand{\Sym}{\operatorname{Sym}}

\newcommand{\la}{\langle} 
\newcommand{\ra}{\rangle}      

\newcommand{\Sub}{\texttt{Subgroups}}
\newcommand{\QSub}{\texttt{Quasi-Subgroups}}
\newcommand{\Bases}{\texttt{Bases}}
\newcommand{\bases}{\texttt{Bases}}
\newcommand{\QBases}{ \texttt{Quasi-Bases}}
\newcommand{\Subgroup}{\texttt{Subgroup}}

\newcommand{\st}[1]{\la{#1}\ra_*}

\maketitle 
\centerline{This paper has been accepted for publication in Rend. Sem. Mat. Univ. Padova}

\begin{abstract}
\noindent
Let $G$ be a  torsion--free abelian group of finite rank.  The orbits of the action of $\Aut(G)$ on the set of maximal independent subsets of $G$ determine the indecomposable decompositions.  $G$ contains a direct sum of pure strongly indecomposable groups as a   subgroup of finite index.
\end{abstract}

%\begin{classification}

%\begin{keywords}

%\end{keywords}

\section{Introduction}
Let $G$ be a finite rank subgroup of    $V=\bbQ^{\bbN}$. In the first part of this paper, Sections 2, 3 and 4, I study the action of  $\Aut(G)$   on the maximal independent subsets of $V$ contained in $G$.  I show that the orbits of this action  determine the isomorphism classes of  indecomposable direct decompositions of $G$.

In the second part, Sections 5 and 6, I study  the action of  $\Aut(G)$ on the set of strongly indecomposable quasi--decompositions of $G$.  Each strongly indecomposable quasi--decomposition   determines an isomorphism class of  subgroups   of $G$  of finite index which are direct sums of strongly indecomposable pure subgroups. 
 
Finally, in Section 7, I initiate a programme to classify strongly indecomposable groups.

Among other  new results of this paper are a group theoretic proof of Lady's Theorem (Corollary \ref{Jtheorem}) which states  that $G$ has only finitely many non--isomorphic   summands, and an extension of the notion of regulating subgroup from   acd groups   to all finite rank groups, (Remark \ref{6.12}, (2)).
 \section{Notation }

Let $V=\bbQ^{\bbN}$   and denote by $\calV$ the set of finite rank additive subgroups of $V$.
   Given $G\in\calV$ and $S\subseteq V$, let   $\la S\ra $ be  the subgroup of $V$ generated by $S$,
  $[S] $ the  subspace of $V$ generated by $S$ and    $S_*= [S]\cap G$.
If $S\subseteq G,\ S_*$ is just   the pure subgroup of $G$  generated by $S$,  but in general we do not insist that $S\subseteq G$.
  
Note that $ \la S\ra$ is the group of all finite   {integral combinations} of $S$, 
 and $[S]$ is  the vector space of all finite   {rational combinations} of $S$. 
Since $S\subseteq V$ is integrally independent   if and only if $S$ is rationally independent,    we generally omit the adjective.

We identify endomorphisms and automorphisms of $G$ with their unique extensions to the vector space $[G]$.

If $ r\in\bbQ^*$, the non--zero rationals, the statement $r=a/b$ will imply that $a\in\bbZ^*$, the non--zero integers,  $b\in\bbN$, the natural numbers and $\gcd(a,\,b)=1$.

A \textit{type} is a group $\tau$ satisfying  $\bbZ\leq\tau\leq\bbQ$. Since $\tau/\bbZ\leq \bbQ/\bbZ\cong\prod_{p\in\bbP}\bbZ(p^\infty)$, 
$\tau/\bbZ$ is a torsion group of  $p$--rank at most 1 for each prime $p$.

 Let $a\in G\in\calV$. The \textit{type of $a$} in $G$, \[\type_G(a)=\{r\in\bbQ\colon ra\in G\}\] which is clearly a type.\footnote{This definition of $\type_G(a)$  is not the standard one, \cite[\S 85]{F2}, but is equivalent to it, as shown in \cite[\S 2.2]{M}.} 
When there is no ambiguity, we omit the subscript $G$.

  A  maximal independent subset of $G\in\calV$ is called a \textit{basis of $G$}.
 Let $\Bases(G)$ denote the set of bases of $G$. In particular, $\bases(G)\subseteq\Bases([G])$,
 the set of vector space bases of   $[G]$. 
 It is well known, see for example \cite[Theorem 16.3]{F1}, that   $\rank(G)$ is the cardinality of any  basis.  so $\rank(G)=\dim([G])$. The following proposition shows how the groups $\la B\ra,\  B_*$ and $[B]$, where $ B\in\Bases(G)$, are related.

\begin{proposition}\label{Prop1} Let   $ G\in\calV$ with $\rank(G)=k$ and let $B\in\Bases(G)$. 
 
\begin{enumerate}  
\item $\la B\ra$ is a free  subgroup of $G$ of rank $k$;
\item $[B]=[G]$ is a subspace of $V$ of dimension $k$;
\item  $B_*=G$;
\item   $G/\la B\ra$ and $[B]/G$ are   torsion groups, the latter being divisible.
 \end{enumerate}
\end{proposition}
\begin{proof} (1) Each $a\in\la B\ra$ has a unique representation as  $a=\sum_{b\in B}n_bb,\ n_b\in\bbZ$.  

(2) $B$ is maximally independent in $[G]$.

(3) By definition, $  B_*\leq G$. Since   $[G]=[B],\ G\leq [B]\cap G= B_*$.

(4) $[B]/\la B\ra\cong\left(\bbQ/\bbZ\right)^k$ is a  torsion group;  $G/\la B\ra$ is a subgroup and $[B]/G$ a factor group.
  \end{proof}

 To justify the name basis,  we note that 
 bases of groups share several properties with bases of  vector spaces. In particular, they are independent spanning sets in the following sense:
 \begin{proposition}\label{bassisprops}  Let  $G\in\calV$  and $0\ne H\in\Sub (G)$. Then: 
 \begin{enumerate} \item Every basis of $G$ is a basis of $[G]$;
 \item  for every   $B\in\Bases([G])$, there is a minimum $m\in\bbN$ such that   $mB=\{mb\colon b\in B\}\in\Bases(G)$;
\item $H$ has a basis   of cardinality $\ell\leq \rank(G)$;  
   \item Every basis of $H$    extends to a basis of $G$;
  \item If   $B\in \Bases(G)$, then every $a\in G$ has a unique representation as $k^{-1}\sum_{b\in B}n_bb$ where $n_b\in\bbZ,\ k\in\bbN$ and $\gcd\{k,\,n_b\colon b\in B\}=1$.
 \end{enumerate}
 \end{proposition}
\begin{proof} (1) If $B\in\Bases(G)$ , then $B$ is a maximal independent subset of $[G]$.

(2) Since $[G]/G$ is torsion, each $b\in B$ has finite order, say $m_b$, modulo $ G$. Let $m=\lcm\{m_b\colon b\in B\}$. Then $m$ is minimal such that $mB\in\Bases(G)$.

(3) $[H]\leq[G]$ and (1) imply $\ell\leq \rank(G)$.

(4) Let $C$ be a basis of $H$. Then $C$ is independent in $G$ and hence extends to a basis $B=C\cup D$. Hence there is a least $m\in\bbN$ such that $B=C\cup mD$ is a basis of $G$.

(5) Let $a\in G$. Since $B\cup\{a\}$ is integrally dependent, there exists a least $k\in\bbN$ such that $ka=\sum_{b\in B}n_bb\colon n_b\in\bbZ$. Hence $a=k^{-1}\sum_{b\in B}n_bb\in V$ and $\gcd\{k,\,n_b\colon b\in B\}=1$.
\end{proof}
 
We call the expression  $k^{-1}\sum_{b\in B}n_bb$ the \textit{$B$--representation of $a$}.

\section{Bases and Decompositions} 

To simplify the notation, from now on \lq decomposition\rq\  of a group means non--trivial direct decomposition and \lq partition\rq\ of a set means  partition into non--empty subsets.

Let $B\in \Bases(G)$ and let  $C\dot\cup D$ be a   partition of $B$. We say that $C\dot\cup D$ is a \textit{splitting partition of $B$}, and $B$ is a \textit{splitting basis},  if $G=C_*\oplus D_*$, while $B$ is an  \textit{indecomposable basis}  if $B$ has no splitting partition.

For clarification, note that 
$G$ can have both splitting and non--splitting bases. For example, let $G=\bbZ\oplus \bbQ$.  Then $B_1=\{(1,0),\,(0,1)\}$ is a splitting basis but $B_2=\{(1,0),\,(1,1)\}$ is a basis that is not splitting. However,  Proposition \ref{dirdec} shows that if $G$ is indecomposable then all bases are indecomposable. 

\begin{proposition}\label{dirdec}   Let $G\in\calV$ and $B\in\Bases(G)$. Then:
\begin{enumerate} \item For any partition $B=C\dot\cup D,\  C_*\cap D_*=\{0\}$, 
and   $C\dot\cup D$ is a splitting partition of $B$  if and only if $C_*+ D_*$  is pure in $G$.
\item  $G=H\oplus K$ if and only if $B$ has a splitting partition $B=C\dot\cup D$ with $H=C_*$ and $K=D_*$.
\end{enumerate}
\end{proposition} 
\begin{proof} (1) If $B=C\dot\cup D$ is any partition of $B$, then $[C]\cap [D]=0$, so $C_*\cap D_*=0$. Since $\rank C_* +\rank D_*=\rank(G),\ C_*\oplus D_*=G$ if and only if  $C_*+ D_*$ is pure in $G$.

(2) $ (\Rightarrow )$ Let $C\in\Bases(H)$ and $D\in\bases(K)$. Then $H=C_*$ and $K=D_*$ so $B=C\dot\cup D$ is a splitting basis for $G$.

$ (\Leftarrow ) $
Since $ C_*\oplus  D_*= G$ with $C\in\Bases(H)$ and $D\in\Bases(K),\ G=H\oplus K$.
\end{proof}

   Proposition \ref{dirdec}   is most useful in the contrapositive, which we state for future reference.
\begin{corollary} \label{contra} Let $G\in\calV$. Then
  $G$  is indecomposable if and only if 
for all $B\in\Bases(G)$ and for all partitions $B=C\dot\cup D,\ C_*\oplus D_*$ is a subgroup of $G$ with non--zero 
 torsion quotient.
\qed
\end{corollary}

It is now routine to extend these results to complete decompositions of $G$. Let $B\in\Bases(G)$, and let $\calB=\dot\cup_{i\in[t]}B_i$ be a splitting partition of $B$. Denote the corresponding decomposition  $\bigoplus_{i\in[t]} (B_i)_*$ of $G$ by $G(\calB)$.

Splitting partitions $\calB=\dot\cup_{i\in[t]} B_i$ and $\calC=\dot\cup_{j\in[s]}C_j$ of $B\in\bases(G)$ are isomorphic, denoted $\calB\cong\calC$, if $t=s$ and there is a permutation $\pi$ of $t$ and isomorphisms $\alpha_i$ such that each $(C_{i*})\alpha_i=(B_{i\pi})_*$.

Let $G\in\calV$ and $B\in\Bases(G)$. A  decomposition $G=\bigoplus_{i\in[t]} A_i$ is \textit{complete}  if each $A_i$ is  indecomposable.  A   partition $\calB=\dot\cup_{i\in[t]} B_i$ of $B$ is 
  a \textit{complete  splitting  partition}  if  $G(\calB)=\bigoplus_{i\in[t]}(B_i)_*$   is a complete   decomposition of $G$. 
   
\begin{proposition}\label{extension} Let $G,\, A_i\colon i\in[t] \in\calV$.  The following are equivalent:
\begin{enumerate}
\item $G=\bigoplus_{i\in[t]} A_i$ is a complete decomposition;
\item For each $B_i\in\Bases(A_i),\ B=\dot\cup B_i$ is a complete  splitting partition such that for all $i\in[t],\ A_i=B_{i*}$;
\item $G$ has a basis $B=\dot\cup_{i\in[t]} B_i$ such that $\bigoplus_{i\in[t]}(B_i)_*$ is a complete decomposition which is pure in $G$.
\end{enumerate}
\end{proposition}
\begin{proof}  By Proposition  \ref{dirdec}, for all parts of the proposition, the statement holds  if $t=2$. 

Assume that each statement holds if $t=k$ and let $t=k+1$. Let  $H=\bigoplus_{i>1}A_i$, so $G=A_1\oplus H$. Then all parts hold with $G$ replaced by $H$, and hence by Proposition  \ref{dirdec}, all parts hold for $G$.
\end{proof}

\section{Automorphisms of $G$}

We first note without proof some well known properties of $\Aut(G)$. For any $\alpha\in\Aut(G)$ and any set $S\subseteq G,\   S\alpha$ denotes the set $\{s\alpha\colon s\in S\}$. 
\begin{lemma}\label{properties} Let $\alpha\in\Aut(G),\ a\in G,\  S\subseteq G,\ H$  and $K\in\Sub(G) $ and $B\in\Bases(G)$. 
\begin{enumerate} \item  $S_*\alpha=(S\alpha)_*$;
\item $H\cap K=0$ if and only if $H\alpha\cap K\alpha=0$;
 \item $\type(a)=\type(a\alpha)$
 \item Let  $r\in\bbQ,\ n_b\in\bbZ$ for all $b\in B$. Whenever either side is defined, so is the other and $(r\sum_{b\in B} n_bb)\alpha=r\sum_{b\in B} n_b(b\alpha)$.
 \qed\end{enumerate}
 \end{lemma}
 
\begin{proposition}\label{4.2} $\Aut(G)$ acts on $\bases(G)$.    This action preserves splitting partitions, indecomposable bases,  and complete splitting partitions.
 \end{proposition}
 
 \begin{proof} Let $B\in\bases(G)$, and $\alpha\in\Aut(G)$. Then Lemma \ref{properties} (4)  implies that $B\alpha\in\bases(G)$.  Clearly, $B1_G=B$   and for all $\alpha,\ \beta\in\Aut(G),\ (B\alpha)\beta= B(\alpha\beta)$.
 
 If $B=C\dot\cup D$ and  $G=C_*\oplus D_*$, then $B\alpha=C\alpha\dot\cup D\alpha$ and $G=C_*\alpha\oplus D_* \alpha$. Conversely, if $B$ has no splitting partition, then $B\alpha$ has no splitting partition.
 
 Let $\calB=\dot\cup_{i\in[t]}B_i$ be a complete splitting partition of $B$, so $G(\calB)$ is  a complete decomposition. Then $G= \bigoplus_{i\in[t]} ( B_i\alpha)_*$ is also a complete   decomposition, so $\calB\alpha=\dot\cup_{i\in[t]} B_i\alpha$ is a complete splitting partition of $B\alpha$.
 \end{proof}

\begin{corollary}\label{4.3}  Let $G\in\calV$. The following are equivalent:
\begin{enumerate}\item $ \bigoplus_{i\in[t]}A_i$ is a complete decomposition of 
 $G$;
\item  $G$ has a basis with complete  splitting  partition $B=\dot\bigcup_{i\in [t]} B_i$ with $B_i\in\Bases(A_i)$;
\item For all    $\alpha\in\Aut(G)\ \dot\bigcup_{i\in[t]} B_i\alpha$ is a complete splitting partition of a basis of $G$.
\qed\end{enumerate}
\end{corollary}

The action of $\Aut(G)$ on $\Bases(G)$, unlike that of $\Aut([G])$ on $\Bases([G])$,   may be far from transitive; in fact its orbits determine the  isomorphism classes of direct decompositions of $G$.  

We say that complete decompositions $\calD\colon \bigoplus_{i\in[t]}A_i$ and $\calE\colon \bigoplus_{j\in[s]}C_j$ of $G$ are \textit{isomorphic}, denoted $\calD\cong\calE$, if their indecomposable summands are pairwise isomorphic, i.e.   $t=s$ and there is a permutation $\pi$ of $[t]$ such that for all $i\in[t],\ C_{i\pi}\cong B_i$.

\begin{proposition}\label{hilfsatz} Let $\calB=\dot\bigcup_{i\in[t]} B_i,\ \calC=\dot\bigcup_{j\in[s]} C_j$ be complete partitions of $B\in\Bases(G)$.  There  exists $\alpha\in\Aut(G)$ such that $\calB\alpha=\calC$ if and only if $G(\calB)\cong G(\calC)$.
\end{proposition}

\begin{proof}  
Assume there exists such $\alpha\in\Aut(G)$. Proposition \ref{hilfsatz} implies that for each $i\in[t]$, there is a $j\in[s]$ such that  $(C_j)_*=(B_i\alpha)_*$ and this correspondence is 1--1. Hence $G(B)\cong G(C)$.

Conversely, if for all $i\in[t],\ \alpha_i\colon ( B_i)_* \to ( C_{i\pi})_*$ are isomorphisms, then $\alpha=(\alpha_i\colon i\in[t])\in\Aut(G)$.
  \end{proof}

\begin{theorem}\label{main1} 
The orbits of $\Aut(G)$ acting on the complete decompositions of $G$ are the  isomorphic complete decompositions.
\end{theorem}

\begin{proof} Let $\calD,\ \calE $ be  complete decompositions of $G$. Then by Proposition \ref{hilfsatz} $\calD\cong\calE$ if and only if there exists $\alpha\in\Aut(G)$ such that $\calD\alpha=\calE$.

Thus every orbit consists of isomorphism classes of complete decompositions, and every isomorphic pair of complete decompositions are in the same orbit.   
\end{proof}

To clarify Theorem \ref{main1}, note that in general, $G$ may have several non--isomorphic complete decompositions, each of which determines several complete splitting partitions of bases of $G$. For each isomorphic   pair $\calB,\,\calC$ of complete splitting partitions of bases,  there may be several $\alpha\in\Aut([G])$ such that $\calB\alpha=\calC$. 

\section{The Quasi  Category of $\calV$}\label{5}

Let $G\in\calV$. There is  a class of subgroups of $V$   which shed  light on the structure of $G$.

\begin{definition} Let  $H,\ G\in\calV$. $H$ is \begin{itemize}
\item   \textit{quasi--equal to   $G$}, written $H\dot= G$,  if there exists $r\in\bbQ^*$ such that $rH= G$.\footnote{This definition, as well as the notation, differs from that in  \cite[\S92]{F2}. However, it is more suited to our context and the two definitions are equivalent.}

 \item \textit{quasi--isomorphic to $G$}, written $H\approx G$,  if $H$ and $ G$ are isomorphic to quasi--equal subgroups of $V$.
\end{itemize} \end{definition}

Quasi--equal groups may have very different structures. Fuchs \cite[Example 2, \S88]{F2} presents examples of groups $G\dot=H$ of arbitrary finite rank $n\geq 2$ such that $G$ is completely
decomposable and $H$ is indecomposable.

The properties of these relations are summarised in the following proposition, whose proof is routine.

\begin{proposition}\label{qprop}\begin{enumerate}

\item Quasi--equality and quasi--isomorphism are equivalences  on $\calV$ which extend equality and isomorphism respectively;
\item   $(a/b)H= G$ if and only if  $aH=bG\leq H\cap G\dot=G$;

\item   $H\dot= G$ implies that $[H]=[G]$ and $H\approx G$ implies $[H]\cong [G]$.
\item If $A$ and $B$ are pure subgroups of $G$ with $A\dot=B$, then $A=B$.
\qed\end{enumerate}
\end{proposition}
\begin{notation}\label{quasinote}
Let $G\in\calV$.
\begin{itemize}
\item A  \textit{quasi--decomposition} $G$ is a  quasi--equality  $G\dot=\bigoplus_i A_i$; the groups $A_i$ are called \textit{quasi--summands} of $G$;

\item $G$  is \textit{strongly indecomposable} if it has no proper quasi--decompositions;

\item A  \textit{strong decomposition} is a direct sum of strongly indecomposable groups, and a \textit{strong quasi--decomposition of $G$}
is a  strong decomposition quasi--equal to $G$.
\end{itemize}

Completely decomposable groups are the type examples of strong decompositions and   almost completely decomposable  (acd) groups of  strong quasi--decomposition. In these cases, the direct summands of a strong (quasi--)decomposition are of rank 1.

\end{notation}

 \begin{definition}

 The \textit{quasi--automorphism group} of $G$   \[ \bbQ^*\Aut(G) :=\{r\alpha\colon   r\in \bbQ^*,\ \alpha\in\Aut(G)\}.\] 
\end{definition}

The properties of  $\bbQ^*\Aut(G)$ are summarised in the following proposition, whose proof follows immediately from the definitions:

\begin{proposition}\label{5.7} For all $G\in\calV$,
\begin{enumerate}
\item $\Aut(G)\leq\bbQ^*\Aut(G)\leq\Aut([G])$;
\item $\Aut(G)$ is a normal subgroup of $\bbQ^*\Aut(G)$;

\item  $H\dot=G$ if and only if there exists   $r\alpha\in\bbQ^*\Aut([G])$ such that $rG\alpha =H$;
 \item $\bbQ^*\Aut(G)$ acts   on the following sets:

  strongly indecomposable subgroups of $G$;
quasi--decompositions of $G$;
 strong  quasi--decompositions of $G$.
\qed \end{enumerate}
 \end{proposition}

The notions of quasi--equality and quasi--isomorphism are due to   \cite{J1}  and  \cite {J2}  and were put into a categorical context by \cite{Walker}. The    properties of this category are outlined in \cite[\S7]{Arnold}.

The most important such property  is the existence and uniqueness of  strong  quasi--decompositions.

\begin{theorem}
\label{Jonsson}  $[$\textbf{J\'onsson's Theorem}$]$, \cite[Theorem 92.5]{F2}  Let $G\in\calV$. Then  $G$ has   strong quasi--decompositions.

If $ \bigoplus_{i\in[t]}A_i$  and $ \bigoplus_{j\in[m]}C_j$ are strong quasi--decompositions of $G$, then $t=m$ and there is a permutation $\pi$ of $[t]$ such that for all $i,\ A_i\approx C_{i\pi}$. 
\end{theorem}

 \section{J\'onsson Bases}    
        
 Throughout this section, $0\ne G\in\calV$.  Let
 $\calJ $  be a  set of  representatives in $\calV$ of the isomorphism classes of the strongly indecomposable quasi--summands  of $G$. J\'onsson's Theorem implies that $\calJ$ is finite. Let $\Rep(G)$ be the set of all such $\calJ$.

Recall that endomorphisms of $G$ are identified with their extensions to $[G]$ and that for each  $J\in\calJ,\    J_*=[J]\cap G$. 
Recall too that set partitions and group decompositions are assumed to be non--trivial.

\begin{notation}\label{quo}  \begin{enumerate}\item A \textit{J\'onsson basis} of $G$ is a strong decomposition  $\bigoplus_i A_i$
 of finite index in $G$ such that each $A_i$ is pure in $G$.  
 \item $\Jon(G)$ is the set of all J\'onsson bases  of $G$. 
 \item If $A\in\Jon(G)$, the finite group $G/A$ is a \textit{J\'onsson quotient} of $G$.
\end{enumerate}  \end{notation}

\begin{remark}\label{remark} \begin{itemize}\item If $G$ is acd, $\Jon(G)$ consists   of   full completely decomposable subgroups;
\item  By Proposition \ref{5.7} (4),  $\Aut(G)$ acts on  ((strongly) indecomposable) summands.  The following proposition shows that  $\Aut(G)$ acts on J\'onsson bases.\end{itemize} \end{remark}

 \begin{proposition}\label{6.4}
 Let $\calJ\in\Rep(G)$. \begin{enumerate}\item For all $J\in\calJ$, there exists a maximum $n_J\in\bbN$ such that $\bigoplus_{J\in \calJ}(J_*)^{n_J}\in\Jon(G)$;
 \item  For all $A\in\Jon(G)$ and all $\alpha\in\Aut(G),\ A\alpha\in\Jon(G)$;

\item  For all $A\in\Jon(G)$, there exists $\alpha\in\Aut(G)$ such that $A\alpha=\bigoplus_{J\in\calJ} (J_*)^{n_J}$.
\end{enumerate}

 \end{proposition}

\begin{proof} (1) Let $n_J$ be the number of isomorphic copies of $J$ which occur in some strong decomposition of finite index in $G$. By J\'onsson's Theorem, $n_J$ is uniquely determined.
The group $\bigoplus_{J\in \calJ}(J_*)^{n_J}$ is a strong decomposition of finite index   in   $ G$ in which each summand $J_*$ is a pure subgroup of $G$. 

(2)  If  $A=\bigoplus_{i\in[t]}A_i$, then $A\alpha=\bigoplus_{i\in[t]}(A_i\alpha)$, a strong decomposition of finite index in $G$.

  (3)  For all $A=\bigoplus_{i\in[t]}A_i\in\Jon(G)$,      there exists $\alpha_i\in\Aut(A_i)$, a partition $\calJ=\dot\cup \calJ_i$ and $m_J\in[0, n_J]$ with $\sum m_j=n_j$, such that $A_i\alpha_i=\bigoplus_{J\in \calJ_i}(J_*)^{m_J}$. Take $\alpha=\sum_i\alpha_i$.   Then $A\alpha = \bigoplus_{J\in \calJ}(J_*)^{n_J}$
\end{proof}

We now show the relation between decompositions of $G$ and  $\Jon(G)$. To clarify the notation,  the decompositions of $A\in\Jon(G)$ in the following proposition, are not  necessarily their  decompositions into pure strongly indecomposable summands described in Notation 6.1 (1).

 \begin{proposition}\label{TFAE}  
 \begin{enumerate} 

\item If $G=\bigoplus_{i\in[t]} H_i$ and $A_i\in\Jon(H_i)$, then $\bigoplus_{i\in[t]} A_i\in\Jon(G)$;
\item Let $A=\bigoplus_{i\in[t]} B_i\in\Jon(G)$, then $\bigoplus_{i\in[t]}B_{i*}\dot=G$ and  $\bigoplus_{i\in[t]}B_{i*} =G$ if and only if $\sum_{i\in[t]}B_{i*}$ is pure in $G$.
\item If $A=\bigoplus_{i\in[t]} B_i\in\Jon(G)$, then for all $J\in\calJ$, there exists a partition $n_J=\sum_{i\in[t]} (m_{i})_J$  (where we allow some terms to be 0),  such that $B_{i*}\cong \left(\bigoplus_{J\in\calJ}(J_*)^{(m_{i})_J}\right)_*$.
\end{enumerate}  

\end{proposition} 
\begin{proof} (1)  $\bigoplus_{i\in[t]} A_i$ is a strong decomposition of finite index in $G$ whose summands are pure in $H_i$ and hence in $G$.

(2) The groups $B_{i*}$ are disjoint pure subgroups of $G$ whose sum has   finite index in $G$. Hence $\sum_{i\in[t]}B_{i*}$ is pure in $G$ if and only if $G=\bigoplus_{i\in[t]}B_{i*} $.

(3) By J\'onsson's Theorem, each $B_i$  is isomorphic to a direct sum of  elements of the set $\calJ$ and their sum accounts for all of them. Hence by  and Proposition \ref{qprop} (3), $B_{i*}\cong \left(\bigoplus_{J\in\calJ}(J_*)^{(m_{i})_J}\right)_*$ for some ${(m_{i})_J}\in[0, n_J]$ with $\sum_{i\in[t]} (m_{i})_J=n_J$. 
\end{proof}

\begin{corollary}\label{Jtheorem}   
 \textbf{Lady's Theorem}, \cite[Theorem 6.9]{Fuchs}
Let $G\in\calV$. Then $G$ has only finitely many non--isomorphic  summands.\footnote{In his recent edition of \lq Abelian Groups\rq\ \cite[Lemma 6.8]{Fuchs} Fuchs  states that since Lady's Theorem is one of the most important results in torsion--free abelian group theory, a group--theoretical proof would be most welcome. Our proof replaces  the Jordan--Zassenhaus Lemma in Lady's proof by J\'onsson's Theorem, whose proof,  while still ring theoretical, is rather more transparent.}
 \end{corollary}
\begin{proof}  If $H$ is a summand of $G$ and $\calJ\in\Rep(G)$, then By Proposition \ref{TFAE},  $H=A_*$ for some    $A\cong \bigoplus_{J\in\calJ}(J_*)^{m_j},\ m_j\in[0,n_j]$.
 There are only finitely many choices for $J$ and   $m_J$.
\end{proof}

\begin{notation}\label{splitt}Let $A\in\Jon(G)$.
\begin{itemize}\item 
A decomposition  $A=\bigoplus_{i\in[t]}A_i$  is a \textit{splitting decomposition}  of $A$ if $G=\bigoplus_{i\in[t]}A_{i*}$;

\item $A$ is \textit{non--split} if $A$ has no splitting decomposition;
\item  $A=\bigoplus_{i\in[t]}A_i$ is a  \textit{complete splitting decomposition} if it is a splitting decomposition and each $A_i$ is non--split.
\end{itemize}

\end{notation}

Propositions \ref{6.4} and \ref{TFAE} have the following immediate corollaries:

\begin{corollary}\label{converse}\begin{enumerate}
\item  $G$ is indecomposable if and only if every  $A\in\Jon(G)$ is non--split;

\item If $G=\bigoplus_{i\in[t]} H_i$ then for all $A_i\in\Jon(H_i),\ \bigoplus_{i\in[t]} A_i$ is a  splitting decomposition of $A\in\Jon(A)$. The decomposition of $G$   is complete if and only if the splitting decomposition of $A$ is.
\qed\end{enumerate}
\end{corollary}

\begin{theorem} \label{actsonindec}$\Aut(G)$ acts on complete  decompositions of $G$ and on complete splitting J\'onsson bases. In both cases, the orbits are the isomorphism classes of complete  decompositions and complete splitting   J\'onsson bases. 
\end{theorem}
\begin{proof}

Let $\calD\colon  \bigoplus_{i\in[t]}H_i$ be an indecomposable decomposition of $G$ and  $\bigoplus_{i\in[t]} A_i$ a corresponding complete splitting J\'onsson basis.  Let $\alpha\in\Aut(G)$. Then
$\bigoplus_{i\in[t]}H_i\alpha$ is also an indecomposable decomposition and $\bigoplus_{i\in[t]} A_i\alpha$ the corresponding complete splitting J\'onsson basis. 

On the other hand, let $\calE\colon \bigoplus_{i\in[t]}K_i$ be an indecomposable decomposition of $G$ with each $K_i$ isomorphic to a unique $H_i$ by some isomorphism $\alpha_{i}$.
Then $\alpha=(\alpha_i\colon i\in [t]) \in\Aut(G)$, so $\calE$ is in the orbit of $\calD$. Thus the orbit of $\calD$ under $\Aut(G)$ consists of isomorphic indecomposable decomposition.

The argument extends to the corresponding J\'onsson bases.
\end{proof}
\begin{remark} It is clear that each decomposition of $G$ refines to a complete decomposition,  and each splitting decomposition of $A\in\Jon(G)$ refines to a complete splitting decomposition. In neither case is the refinement necessarily unique.
\end{remark}
\subsection{J\'onsson Quotients}

Let $G\in\calV$ and $A\in\Jon(G)$.
    Let $\eta\colon G\to G/A$ be the natural surjection,  and let $  U\oplus   W$ be a decomposition of $ G/A$. We say $U\oplus W$ \textit{lifts  to $G$} if there exists a decomposition $G=K\oplus L$ such that $U=  K\eta$ and $W=  L\eta$.

 \begin{proposition}\label{identifies}  $G/A=U\oplus W$ is a decomposition lifting to $G=K\oplus L$ if and only if $A$ has a splitting decomposition $B\oplus C$ such that $U\cong B_*/B$ and $W=C_*/C$.
 \end{proposition}
 \begin{proof} $(\Leftarrow)$ If such a splitting decomposition exists, then by Propossition \ref{TFAE} (2), $G=B_*\oplus C_*$ so $G/A=B_*/B\oplus C_*/C$..
 
 $(\Rightarrow)$ Suppose $G=K\oplus L$ with $U=K\eta$ and $W=L\eta$. Let $B=K\cap A$ and $C=L\cap A$, so $B\cap C=0$. Let $a\in A$, say $a=k+\ell$ with $k\in K$ and $\ell\in L$. Since $a\eta=0$, if $k\eta\ne 0$, then $0\ne \ell\eta=-k\eta\in U\cap W$, a contradiction. Hence 
  $B+C=A$. Thus $B\oplus C$ is a splitting decomposition of $A$ and $G/A=U\oplus W$ lifts to $B_*\oplus C_*=G$.
 \end{proof}

\begin{remark}\label{6.12} \begin{enumerate}\item  If $G$ is indecomposable, then   \cite[Example 16.8.11]{M} shows that $G\eta$ need not be.  However,  if $U$ is an indecomposable summand of $G\eta$ then  $U\eta^{-1}$ is indecomposable in $G$.

\item Since J\'onsson quotients are finite, there is a least index $e\in\bbN$ and $A\in\Jon(G)$ such that $|G/A|=e$. Such $A$ are called \textit{regulating J\'onsson bases} since in the case that $G$ is almost completely decomposable, they are regulating subgroups as defined in a different way in \cite[\S 4.2]{M}. Thus regulating J\'onsson bases are an extension to $\calV$ of regulating subgroups of acd groups.  Their properties remain to be investigated.
\end{enumerate}
\end{remark}

\begin{proposition}\label{related}  Let $G\in\calV,\ A\in\Jon(G)$ and  $T=G/A$. 
Then $T=\bigoplus_i S_i$ is a  decomposition which lifts to   a complete decomposition $\bigoplus_i H_i$ of $G$  if and only if   $A=\bigoplus_i C_i$ is a complete splitting decomposition of $A$ with $C_{i*}\cong H_i$ and each $S_i$ has no decomposition which lifts to $G$.
\end{proposition}

\begin{proof}   The statement follows Proposition \ref{identifies} and a routine induction on the number of summands.
 \end{proof}

We turn  now to the  action of $\Aut(G)$ on the set   of J\'onsson quotients of $G$. The proof of the following lemma is a routine application of the definitions.

\begin{lemma}\label{action of Aut} Let $A\in\Jon(G)$ and $\alpha\in\Aut(G)$. Define $\ov\alpha\colon G/A\to G/(A\alpha)$ by $\ov\alpha\colon g+A\to g\alpha+A\alpha$. Then $\ov \alpha$ is an isomorphism and $(G/A)^\alpha=(G/A)\ov\alpha$ is a group action.
\qed\end{lemma}
\begin{proposition}\label{cor} \begin{enumerate}\item  If $\eta\colon G\to G/A$ is the natural epimorphism, then \hfil\break $\alpha^{-1}\eta\alpha\colon G\to G/A\alpha$ is the natural epimorphism.
\item If $G/A=S\oplus T$ and this decomposition lifts to $G$, then $G/A\alpha=S\ov\alpha\oplus T\ov\alpha$ and this decomposition also lifts to $G$.

\item The map $\alpha\to\ov\alpha$ is a (non--abelian) group homomorphism. Its image is the group of automorphisms of $G/A$ which lift to $G$, and its kernel is $\{\alpha\in\Aut(G)\colon \alpha-1\in n\End(G)\}$ where $n=\exp(G/A)$.
\end{enumerate}\end{proposition}

\begin{proof} (1) follows from the definition of the group action.

(2) Let $\mu\colon G/A\to S$ be the projection determined by the first decomposition. Then $\mu'=\ov\alpha^{-1}\mu\ov\alpha\colon G/A\alpha\to S\ov\alpha$ is the projection determined by the second decomposition. Since $\mu$ lifts to $G$, say $G=B_*\oplus C_*$ and $A=B\oplus C$, then $A\alpha=B\alpha\oplus C\alpha$ and $G=(B\alpha)_*\oplus (C\alpha)_*$, so  $\mu'$ is the projection which determines the second decomposition.    

(3) This is a routine calculation.
\end{proof}

Being finite,  $G/A$ has, up to isomorphism,  a unique  complete decomposition, which in general   does not lift to $G$. We define a decomposition of $G/A$ which lifts to $G$ to be \textit{unrefinable} if it has no refinement lifting to $G$. 

 \begin{theorem}\label{ } Let $G\in\calV$ and let $\calA$ be the set of
 unrefinable decompositions of J\'onsson quotients of $G$.
  $\Aut(G)$ acts on $\calA$  and the finitely many orbits of this action consist of isomorphic unrefinable decompositions.
 \end{theorem}

\begin{proof}  
Let $\calD\in\calA$, so $\calD=\bigoplus_{I\in[t]} D_i$ an unrefinable decomposition of $G/A$ for some $A\in\Jon(G)$. Then for all $\alpha\in\Aut(G),\ \calD$ is isomorphic to $\calD\alpha:=\bigoplus_{I\in[t]} D_i\ov\alpha$ as described in Lemma \ref{action of Aut}, so all elements of the orbit containing $\calD$ are isomorphic. 
\end{proof}

If $A'\in\Jon(G)$ it is possible that $G/A\cong G/A'$ even if $A$ and $A'$ are in different orbits of $\Aut(G)$.
In any case, the number of isomorphism classes of unrefinable J\'onsson quotients of $G$ is no greater than the finite number of orbits of $\Aut(G)$ acting on complete decompositions of $G$

 \section{Strongly indecomposable groups} \label{SIG}
 
 Strongly indecomposable groups play  a crucial r\^ole in this paper.
I am not aware of any published  classification, but several important properties and examples can be found in \cite[\S 92]{F2} and \cite[\S 9]{Fuchs}. In this section, I present a new characterisation of strongly indecomposable groups. %   and a method for constructing them. 
Without loss of generality, we   assume $G$ is reduced and $\rank(G)>1$. Recall that a strong decomposition in $\calV$ is a direct sum of strongly indecomposable groups.
\begin{notation}\label{lifts} Let $R\leq G\in\calV$.
\begin{itemize}
\item The pair $(R,\,G)$ has \textit{Property $SI$}  if  $R$ is a strong decomposition  such that $G/R$ is an infinite torsion group with no decomposition lifting to $G$.
\item If $B\in\Bases(G)$ then $(B)_*=\bigoplus_{b\in B}b_*$.

\end{itemize}
\end{notation}

A routine calculation shows that
 Property $SI$ is invariant under $\bbQ^*\Aut(G)$.

\begin{theorem}\label{necessity} Let  $G\in\calV$. Then $G$ is strongly indecomposable if and only if  for all  $B \in\Bases(G),\ ((B)_*,\,G)$ 
has property $SI$. 

\end{theorem}
\begin{proof} $(\Longrightarrow)$ Let  $B\in\Bases(G)$, so $(B)_*$ is a strong decomposition and $G/(B)_*$ is torsion. If $G/(B)_*$ is finite, then for some $m\in\bbN,\ mG=(B)_*$, a contradiction, so $G/(B)_*$ is infinite. If some summand of $G/(B)_*$ lifts to $G$ then $G$ is decomposable, another contradiction. Hence $((B)_*\,,G)$ has Property $SI$.

$(\Longleftarrow)$ Suppose that for all $B\in\Bases(G),\ ((B)_*,\,G)$ has property $SI$, but that for some $m\in\bbN,\ mG=H\oplus K$. Let $C\in\bases(H)$ and $D\in\bases(K)$, so that by Proposition \ref{dirdec}, $mG=C_*\oplus D_*$.  Let $B=C\cup D$. Then   $mG/(B)_*$ has a summand $C_*/B_*$ lifting to the summand $C_*$ of $G$, a contradiction.  Hence no basis of $mG$ has a splitting partition, so $G$ is strongly indecomposable.
\end{proof}

%As an example of such a group, let $B$ be a free group of rank $n$, so $[B]/B=(\bbQ)/\bbZ)^n$. Let $\eta\colon [B]\to [B]/B$ be the natural epimorphism and let $G=\eta^{-1} \bbZ(p^\infty)$ for some copy of $\bbZ(p^\infty)$ and $B\in\Bases(G)$.  If $C\in\bases(G)$, then $C=rB$ for some $r\in\bbQ$, and $\ov r\colon G/C\to G/B$ defined by $g+C\mapsto rg+B$ is a well defined group homomorphism. Since $r\bbZ(p^\infty)\cong \bbZ(p^\infty),\ G/C\cong G/B$. Hence $G$ is strongly indecomposable.

A   generalisation of Theorem \ref{necessity}, with essentially  the same proof can be obtained by replacing the strong decomposition $(B)_*$ by an arbitrary full strong decomposition   contained in $G$.
 
\begin{theorem} Let $G\in\calV$. Then $G$ is strongly indecomposable if and only if for all full strong decompositions $D$  contained in $G,\ (D,\,G)$ has property $SI$.
\qed\end{theorem}
 %%%%%%%%%%%%%%%%%%%%%%%%%%%%%%%%%%%%%%%%%%%%%%


\begin{thebibliography}{99}
\bibitem
[\textbf{Arnold}, 1982]{Arnold}
{\ D.~Arnold}, \textit{Finite Rank Torsion Free Abelian Groups and Rings},
Lecture Notes in Mathematics, vol.~931, Springer Verlag, 1982.


\bibitem[\textbf {Fuchs}, 1970]{F1} L. Fuchs, \textit{Infinite Abelian Groups, Vol. 1} Academic Press, 1970.

\bibitem[\textbf{Fuchs}, 1973]{F2} L. Fuchs, \textit{Infinite Abelian Groups, Vol. 2} Academic Press, 1973.


\bibitem[\textbf{Fuchs}, 2015]{Fuchs}
{ L.~Fuchs}, \textit{Abelian Groups}, Springer Monographs in Mathemaics, Springer 2015. 


\bibitem[\textbf{J\'onsson}, 1957]{J1}{B. J\'onsson}, \textit{On direct decompositions of torsion--free abelian groups}, Math. Scand. \textbf {5},  (1957), 230--235.


\bibitem[\textbf{J\'onsson}, 1959]{J2}{B. J\'onsson}, \textit{On direct decompositions of torsion--free abelian groups}, Math. Scand. \textbf {7},  (1959), 361--371.





\bibitem[\textbf{Lady},1974]{Lady}E. L. Lady, \textit{Summands of Finite Rank Torsion--free Abelian Groups},  J. of Algebra, 32,  (1974) 51--52.


\bibitem[\textbf{Mader}, 2000]{M}
{A.~Mader}, \textit{Almost completely decomposable abelian groups},
Gordon and Breach, {\it Algebra, Logic and Applications} Vol.
13, Amsterdam, 1999.

\bibitem[\textbf{Walker}, 1964]{Walker}{E.A. Walker}, \textit{Quotient categories and quasi--isomorphisnms of abelian groups},  Proc. Coll. Abelian groups, Tihany, Akad\'emiai Kiado, 147--162, 1964.



\end{thebibliography}
 \end{document}